\documentclass[10pt]{article}
\usepackage[utf8]{inputenc}
\usepackage[T1]{fontenc}
\usepackage{amsmath}
\usepackage{amsfonts}
\usepackage{amssymb}
\usepackage[version=4]{mhchem}
\usepackage{stmaryrd}
\usepackage{bbold}
\usepackage{graphicx}
\usepackage[export]{adjustbox}
\graphicspath{ {./images/} }
\usepackage{caption}
\usepackage{amsthm}
\usepackage{algorithm}
\usepackage{algpseudocode}
\usepackage{xcolor}

\usepackage{cite}

\newtheorem{dfn}{Definition}[section]
\newtheorem{lem}[dfn]{Lemma}
\newtheorem{thm}[dfn]{Theorem}

\newtheorem{prop}[dfn]{Proposition}

\newtheorem{remark}[dfn]{Remark}
\newtheorem{exm}[dfn]{Example}
\newcommand{\R}{\mathbb{R}}

\allowdisplaybreaks

\title{A Second-Order  Dynamical System for Solving Generalized Inverse Mixed Variational Inequality problems }

\author{Nam Van Tran ${ }^{1 *}$ \\
	${ }^{1}$ Faculty of Applied Sciences,\\ Ho Chi Minh City University of Technology and Engineering,\\ Ho Chi Minh, Vietnam.\\
	*Corresponding author(t). E-mail(t): namtv@hcmute.edu.vn;\\
}
\date{}

\DeclareUnicodeCharacter{27FA}{\ifmmode\Longleftlongrightarrow\else{$\Longleftlongrightarrow$}\fi}

\begin{document}
	\maketitle
	\captionsetup{singlelinecheck=false}

	\begin{abstract}
		In this paper, we study a class of generalized inverse mixed variational inequality problems (GIMVIPs). We propose a novel projection-based second-order time-varying dynamical system for solving GIMVIPs. Under the assumptions that the underlying operators are strongly monotone and Lipschitz continuous, we establish the existence and uniqueness of solution trajectories and prove their global exponential convergence to the unique solution of the GIMVIP. Furthermore, a discrete-time realization of the continuous dynamical system is developed, resulting in an inertial projection algorithm. We show that the proposed algorithm achieves linear convergence under suitable choices of parameters. Finally, numerical experiments are presented to illustrate the effectiveness and convergence behavior of the proposed method in solving GIMVIPs.
	\end{abstract}
	
	
	MSC Classification (2020): 47J20, 65P40 90C30, 37C75.
	
\section{Introduction}

Let $H$ be a real Hilbert space endowed with the inner product $\langle \cdot , \cdot \rangle$ and the induced norm $\|\cdot\|$. Let $K \subset H$ be a nonempty, closed, and convex set. Given a continuous operator $\phi : H \to H$, the \textbf{variational inequality problem (VIP)} consists in finding a point $u^* \in K$ such that
\begin{equation*}\label{pt1}
	\langle \phi(u^*), u - u^* \rangle \ge 0, \quad \forall u \in K.
\end{equation*}
The VIP serves as a powerful and unified framework for modeling a vast array of problems in nonlinear analysis, mathematical economics, and engineering \cite{10, 35}. Over the past decades, VIPs have been extensively studied, leading to the development of numerous numerical schemes, \textbf{particularly those based on projection-type methods} \cite{5, 7, 34}.

In many real-world scenarios, such as transportation equilibria or competitive market models, the underlying operator $\phi$ cannot be observed directly. Instead, the problem must be formulated using an inverse mapping, which gives rise to the \textbf{inverse variational inequality problem (IVIP)}. Let $T : H \to H$ be a single-valued continuous operator. The IVIP associated with $T$ and $K$ seeks a point $w^* \in H$ such that
\begin{equation*}\label{pt4}
	T(w^*) \in K \quad \text{and} \quad \langle v - T(w^*), w^* \rangle \ge 0, \quad \forall v \in K.
\end{equation*}
While IVIPs have recently gained significant attention \cite{2, 11, 17, 18}, modern applications demand even greater modeling flexibility. This led to the introduction of mixed variational inequalities, which incorporate an additional convex function $f$ to account for non-smooth components or regularizations. In this paper, we focus on a significantly more general model: the \textbf{generalized inverse mixed variational inequality (GIMVI)}. The GIMVI consists in finding $w^* \in H$ such that $T(w^*) \in K$ and
\begin{equation}\label{pt10}
	\langle g(w^*), y - T(w^*) \rangle + f(y) - f\big(T(w^*)\big) \ge 0, \quad \forall y \in K,
\end{equation}
where $T, g : H \to H$ are continuous mappings. This framework strictly extends traditional inverse models and provides a robust structure for solving complex equilibrium problems \cite{22, 24, JIT, Nam}.

A highly effective strategy for solving \eqref{pt10} involves the use of continuous-time dynamical systems \cite{1,  19, 31, 33}. Traditional research has predominantly focused on first-order time-varying dynamical systems  \cite{JIT, Nam}, typically formulated as:
\begin{equation*}\label{pt11}
	\dot{w}(t) = \rho \big[ P_K^{\gamma f}\big(T(w(t)) - \gamma g(w(t))\big) - T(w(t)) \big].
\end{equation*}

Although first-order systems are well-established, they often lack the inertial "momentum" required to achieve rapid convergence in complex optimization landscapes. Inspired by the \textbf{``heavy-ball''} method and inertial acceleration principles \cite{9, 37}, second-order dynamical systems---which incorporate an acceleration term $\ddot{w}(t)$---have emerged as a superior alternative. These systems effectively leverage inertial effects to dampen oscillations and accelerate the trajectory toward the equilibrium point \cite{26, 30}. Despite their advantages, the application of second-order dynamical system to the GIMVI problem remains a critical gap in the existing literature, particularly regarding the rigorous proof of global stability and convergence rates.

Motivated by these challenges, we propose the following second-order time-varying dynamical system to address the GIMVI problem:
\begin{equation}\label{pt12}
	\left\{
	\begin{aligned}
		\ddot{w}(t) &+ \kappa(t)\dot{w}(t) + \rho(t)\big[ T(w(t)) - P_K^{\gamma f}\big(T(w(t)) - \gamma g(w(t))\big) \big] = 0, \\
		w(0) &= w_0, \quad \dot{w}(0) = w_1,
	\end{aligned}
	\right.
\end{equation}
where $\kappa(t)$ and $\rho(t)$ are time-varying parameters and $P_K^{\gamma f}$ denotes the generalized $f$-projection operator.

The primary contributions of this work are summarized as follows:
\begin{itemize}
	\item We establish the \textbf{global existence and uniqueness} of the solution trajectories for the proposed second-order system \eqref{pt12} under mild conditions.
	\item By constructing a suitable Lyapunov functional, we prove that the trajectories converge \textbf{exponentially} to the unique solution $w^*$ of the GIMVI problem. This result is particularly significant as it demonstrates that second-order inertial systems can maintain superior convergence rates even for generalized inverse models.
	\item We derive an \textbf{accelerated discrete algorithm} from the continuous model and establish its linear convergence rate, bridging the gap between theoretical dynamical systems and practical numerical implementation.
	\item Numerical experiments are provided to demonstrate that the second-order system exhibits robust performance and rapid error reduction compared to traditional first-order approaches.
\end{itemize}

The remainder of this paper is organized as follows.
Section~\ref{sec.prelim} introduces the essential preliminaries, definitions, and notation used throughout the paper.
In Section~\ref{sec.existence}, we establish the existence and uniqueness of the solution trajectories for the proposed second-order dynamical system.
Section~\ref{hoitumu} is devoted to the analysis of the global exponential stability of these trajectories.
The discrete counterpart of the continuous system is investigated in Section~\ref{roirachoa}, where its convergence properties are analyzed.
Numerical experiments illustrating the effectiveness of the proposed algorithms are presented in Section~\ref{vdso}.
Finally, concluding remarks are provided in Section~\ref{kl}.

\section{Definitions and Preliminaries}\label{sec.prelim}

In this section, we recall several definitions and auxiliary results that will be used throughout the paper.

Let $H$ be a real Hilbert space endowed with the inner product $\langle \cdot , \cdot \rangle$ and the induced norm $\|\cdot\|$.  
We consider the product space $H \times H$, equipped with the inner product and norm introduced in \cite{37}: for all $(x,y), (\bar{x},\bar{y}) \in H \times H$,
\[
\langle (x,y), (\bar{x},\bar{y}) \rangle
:= \langle x, \bar{x} \rangle + \langle y, \bar{y} \rangle,
\qquad
\|(x,y)\| := \sqrt{\|x\|^2 + \|y\|^2}.
\]

We denote by $L_{\mathrm{loc}}^{1}(\mathbb{R}^{+} \cup \{0\})$ the space of functions $f$ such that
$f \in L^{1}([0,c])$ for every $c > 0$.

Next, we recall several notions related to operator properties.

\begin{dfn}\cite{36}
	Let $K$ be a nonempty subset of $H$. An operator $T : H \to H$ is said to be
	\begin{itemize}
		\item[(i)] \emph{monotone} on $K$ if
		\[
		\langle T(w) - T(v), w - v \rangle \ge 0,
		\quad \forall\, w, v \in K;
		\]
		\item[(ii)] \emph{strongly monotone} on $K$ with modulus $\lambda > 0$ if
		\[
		\langle T(w) - T(v), w - v \rangle \ge \lambda \|w - v\|^2,
		\quad \forall\, w, v \in K.
		\]
	\end{itemize}
\end{dfn}

\begin{dfn}
	Let $K$ be a nonempty, closed, and convex subset of $H$.
	Let $T, g : H \to H$ be two operators and let $\alpha > 0$.
	The pair $(T,g)$ is said to be \emph{$\alpha$-strongly coupled monotone} on $K$ if
	\[
	\langle T(w) - T(v), g(w) - g(v) \rangle
	\ge \alpha \|w - v\|^2,
	\quad \forall\, w, v \in K.
	\]
\end{dfn}

When $g$ is the identity operator, the above definition reduces to the classical notion of strong monotonicity.

\begin{dfn}\cite{36}
	An operator $T : H \to H$ is said to be \emph{$\eta$-Lipschitz continuous} on $K$ if there exists a constant $\eta \ge 0$ such that
	\[
	\|T(w) - T(v)\| \le \eta \|w - v\|,
	\quad \forall\, w, v \in K.
	\]
\end{dfn}

In this paper, we employ the generalized $f$-projection operator (see \cite{Wu}), defined by
\begin{equation*}\label{pt3}
	P_{K}^{\gamma f}(w)
	:= \arg\min_{v \in K}
	\left\{
	\gamma f(v) + \frac{1}{2}\|w - v\|^2
	\right\},
\end{equation*}
where $K$ is a nonempty, closed, and convex subset of $H$,
$f : K \to \mathbb{R} \cup \{+\infty\}$ is a proper, convex, and lower semicontinuous function, and $\gamma > 0$.

Observe that if $f$ is the indicator function of $K$, namely
\[
f(w) =
\begin{cases}
	0, & w \in K, \\
	+\infty, & w \notin K,
\end{cases}
\]
then $P_{K}^{\gamma f}$ reduces to the classical metric projection onto $K$.
In the literature, $P_{K}^{\gamma f}$ is also known as the \emph{proximal operator}.

The generalized $f$-projection operator enjoys several useful properties, which are summarized below.

\begin{lem}\label{lm2.1}\cite{35}
	Let $K$ be a nonempty, closed, and convex subset of $H$. Then:
	\begin{enumerate}
		\item[(i)] 
		\[
		\langle w - P_{K}^{\gamma f}(w), P_{K}^{\gamma f}(w) - v \rangle \ge 0,
		\quad \forall\, w \in H, \ v \in K;
		\]
		\item[(ii)] $P_{K}^{\gamma f}$ is nonexpansive, i.e.,
		\[
		\|P_{K}^{\gamma f}(w) - P_{K}^{\gamma f}(v)\|
		\le \|w - v\|,
		\quad \forall\, w, v \in H;
		\]
		\item[(iii)] For $w \in H$, $w^* = P_{K}^{\gamma f}(w)$ if and only if
		\[
		\langle w^* - w, v - w^* \rangle
		+ \gamma f(v) - \gamma f(w^*) \ge 0,
		\quad \forall\, v \in K.
		\]
	\end{enumerate}
\end{lem}

Throughout this paper, we frequently use the following classical inequalities.
For any $w, v \in H$,
\[
\langle w, v \rangle \le \|w\|\,\|v\|,
\qquad \text{(Cauchy--Schwarz inequality)}
\]
and for any $\varepsilon > 0$, Young’s inequality holds:
\[
\|w\|\,\|v\|
\le \frac{\|w\|^2}{2\varepsilon}
+ \frac{\varepsilon \|v\|^2}{2}.
\]

\noindent\textbf{Notation.}
For convenience, we define
\[
A(w(t))
:= T(w(t))
- P_{K}^{\gamma f}\!\left(
T(w(t)) - \gamma g(w(t))
\right).
\]

We now provide a characterization of solutions to the GIMVIP, showing that a point is a solution of the GIMVIP \eqref{pt10} if and only if it is a zero of the operator $A$. 
\begin{prop}  Let $K$ be a nonempty, closed, convex  subset of $H$. Let  $f: K \longrightarrow \mathbb{R} \cup\{+\infty\}$ be a proper, convex, and lower semicontinuous function and $ T: H\longrightarrow H$ be $\eta$-Lipschitz continuous and $\lambda$ be strongly monotone on $K$, and $g: H\longrightarrow H$. Then, $ w^* $ is a solution of GIMVIP \eqref{pt10} if and only if $ w^* $ satisfies	
	\begin{equation*}\label{pt19} 
		T\left( w^* \right)=P^{\gamma f}_{K}\left(T\left( w^* \right)-\gamma  g(w^*) \right).
	\end{equation*}	
\end{prop}	
\begin{proof} Assume that $ w^* $ is the solution of \eqref{pt10}, i.e. $T\left( w^* \right) \in K$ and
	
	$$
	\left\langle\gamma  g(w^*) , y-T\left( w^* \right)\right\rangle+\gamma f(y)-\gamma f\left(T\left( w^* \right)\right) \geq 0, \quad \forall y \in K .
	$$	It follows that 	
	$$
	\left\langle T\left( w^* \right)-\left(T\left( w^* \right)-\gamma g\left(w^*\right) \right), y-T\left( w^* \right)\right\rangle+\gamma f(y)-\gamma f\left(T\left( w^* \right)\right) \geq 0, \quad \forall y \in K .
	$$	
	By Part (iii) of Lemma \ref{lm2.1}, it holds that 
	
	$$
	T\left( w^* \right)=P_{K}^{\gamma f}\left(T\left( w^* \right)-\gamma g\left(w^*\right) \right) .
	$$
	We now suppose that $T\left( w^* \right)=P_{K}^{\gamma f}\left(T\left( w^* \right)-\gamma g\left(w^*\right) \right)$. By Part (iii) of Lemma \ref{lm2.1}, one has 	
	$$
	\left.\left\langle T\left( w^* \right)-\left(T\left( w^* \right)-\gamma g\left(w^*\right) \right)\right), y-T\left( w^* \right)\right\rangle+\gamma f(y)-\gamma f\left(T\left( w^* \right)\right) \geq 0, \quad \forall y \in K .
	$$
	Because $\gamma>0$ it follows that 
	$$
	\left\langle w^* , y-T\left( w^* \right)\right\rangle+f(y)-f\left(T\left( w^* \right)\right) \geq 0, \quad \forall y \in K .
	$$
	Thus, $ w^* $ is a solution of \eqref{pt10}. 
\end{proof}

The lemma below provides some estimates for the operator $A$ defined above. 
	\begin{lem}  Let $K$ be a nonempty convex, subset subset of $H$. Let  $ T: H\longrightarrow H$ be $\eta$-Lipschitz continuous and $\lambda$ be a strongly monotone on $K$, $f: K\longrightarrow \R\cup\{+\infty\}$ be a proper, convex, l.s.c function,  and $g: H\longrightarrow H$ be a $\beta$-Lipschitz continuous operator. Let $ w^* $ be the unique solution of the GIMVIP \eqref{pt10}. Assume that the pair $(T, g)$ is $\alpha$-strongly coupled monotone. For all $\gamma>0$ and $w \in H$,  we define 
		\begin{equation*}\label{dna1} 
			a_{1}:=\frac{a}{(2 \eta+\gamma\beta)^{2}},
		\end{equation*}
		and \begin{equation}\label{dna}
			a:=\left(\lambda+\gamma \alpha-\frac{\eta^{2}}{2}-\frac{\gamma^{2}\beta}{2}-\frac{1}{2}\right)
		\end{equation}
		If $a>0$ then 	
		\begin{equation}\label{pt23} 
			\left\langle w- w^* ,  A(w)\right\rangle \geq a_{1}\| A(w)\|^{2}, 
		\end{equation}	
		\begin{equation}\label{pt23n} 
			a\|w- w^*\|\leq \|A(w)\| ,
		\end{equation}
		and	
		\begin{equation}\label{pt24} 
			\left\langle -A(w), w- w^* \right\rangle \leq-a\left\|w- w^* \right\|^{2}.
		\end{equation}
	\end{lem} 	
	
	\begin{proof}	
		From  \eqref{pt16}, one has 
		\begin{equation}\label{pt25} 
			\|A(w)\|=	\left\|A w-A  w^* \right\| \leq(2 \eta+\gamma \beta)\left\|w- w^* \right\|.
		\end{equation}
		In addition, we have 
		\begin{equation}\label{pt26} 
			\begin{array}{lll}
				& \left\langle A(w), w- w^* \right\rangle \\
				=& \left\langle A(w)-A(w^*), w- w^* \right\rangle \\
				= & \left\langle T(w)-T\left( w^* \right), w- w^* \right\rangle  -\left\langle G(w)-G(w^*), w- w^* \right\rangle .
			\end{array}	
		\end{equation}
		where $G(w):=P_{K}^{\gamma f}( T(w)-\gamma g(w))$. 
		The $\alpha$-strong coupled monotonicity of $ (T, g)$ implies that 
		\begin{equation}\label{pt27} 
			\left\langle T(w)-T\left( w^* \right), g(w)- g\left(w^*\right) \right\rangle \geq \alpha\left\|w- w^* \right\|^{2}.
		\end{equation}	
		Since $ T$ is $\eta$-Lipschitz continuous, $g$ is $\beta$-Lipschitz continuous, using \eqref{pt27}  yields 
		\begin{equation} \label{pt28} 
			\begin{array}{ll}
				& \left\langle G(w)-G(w^*), w- w^* \right\rangle \\
				& \leq \frac{1}{2}\left\|w- w^* \right\|^{2}+\frac{1}{2}\left\|G(w)-G(w^*) \right\|^{2} \\
				& \leq \frac{1}{2}\left\|w- w^* \right\|^{2}+\frac{1}{2}\left\|( T(w)-\gamma g(w))-\left(T\left( w^* \right)-\gamma g\left(w^*\right) \right)\right\|^{2} \\
				& =\frac{1}{2}\left\|w- w^* \right\|^{2}+\frac{1}{2}\left\| T(w)-T\left( w^* \right)\right\|^{2} + +\frac{1}{2} \gamma^{2} \|g(w)-g(w^*)\|^2 \\
				& -\gamma\left\langle T(w)-T\left( w^* \right), g(w)- g\left(w^*\right) \right\rangle \\
				& \leq\left(\frac{1}{2}+\frac{1}{2} \eta^{2}+\frac{1}{2} \gamma^{2}\beta-\gamma \alpha\right)\left\|w- w^* \right\|^{2} . 
			\end{array}
		\end{equation} 	
		Combining \eqref{pt28} and \eqref{pt26}, and using the $\lambda$-strong monotonicity of $T$ we obtain	
		\begin{equation}\label{pt29} 
			\left\langle A(w), w- w^* \right\rangle \geq\left(\lambda+\gamma \alpha-\frac{\eta^{2}}{2}-\frac{\gamma^{2}\beta}{2}-\frac{1}{2}\right)\left\|w- w^* \right\|^{2}=a\|w-w^*\|^2 . 
		\end{equation}	
		It follows that 
		\begin{equation*}
			a\|w-w^*\|^2\leq \|A(w)\|\cdot\|w-w^*\|,
		\end{equation*}
		which implies inequality \eqref{pt23n}.
		
		Now, combining \eqref{pt25} and  \eqref{pt29} with noting that $a>0$ yields 
		
		\begin{equation*} \label{pt30} 
			\begin{array}{ll}
				a_{1}\left\| A(w)\right\|^{2} & =\frac{\left(\lambda+\gamma \alpha-\frac{\eta^{2}}{2}-\frac{\gamma^{2}\beta}{2}-\frac{1}{2}\right)}{(2 \eta+\gamma\beta)^{2}}\left\|A w-A  w^* \right\|^{2} \\
				& \leq a \|w-w^*\|\leq \left\langle A(w), w- w^* \right\rangle .
			\end{array}
		\end{equation*} 	
		The last inequality also implies that  
		$$
		\left\langle -A(w), w- w^* \right\rangle \leq-a\left\|w- w^* \right\|^{2}.
		$$ 
		and hence, \eqref{pt24} holds true. 
	\end{proof}	
		
\section{The existence and uniqueness of solutions}\label{sec.existence} 
 
 In this section, we present conditions under which the existence and uniqueness of the trajectory for the dynamical system \eqref{pt12} are obtained.
 
 We first establish conditions guaranteeing the existence and uniqueness of the solution to the GIMVIP \eqref{pt10}.
 
 \begin{thm}\cite{JIT}\label{unique sol}
 	Let $ g~:  H \longrightarrow H $  be a Lipschitz continuous single-valued mapping with Lipschitz constant $\beta$. Let   \(T: H \longrightarrow H\) be a  $\eta-$ Lipschitz continuous single-valued mapping. Let $ f~: K \longrightarrow \mathbb{R}\cup \{+\infty\}$  be a proper, convex, and lower semicontinuous function.  Assume further that  
 	\begin{enumerate}
 		\item [(i)] $g$ is $\zeta$-monotone on $H$, and $(T, g)$ is $\alpha$-strongly monotone on $H$. 
 	
 		\item [(ii)]   \(\sqrt{\eta^2+\beta^2-2\alpha}+\sqrt{1-2\zeta+\beta^2}<1,\)
 		where $\alpha<\frac{\beta^2+\eta^2}{2}, \zeta < \frac{1+\beta^2}{2}$.  	
 	\end{enumerate}
 Then the GIMVIP \eqref{pt10} has a unique solution.
 \end{thm}
We now define the strong global solution to the dynamical system \eqref{pt12}.  
	\begin{dfn}
 We say that $w:[0,+\infty) \rightarrow \mathcal{H}$ is a strong global solution of \eqref{pt12} if it satisfies the following properties:
 \begin{itemize}
 	\item [(i)] $w, \dot{w}:[0,+\infty) \rightarrow \mathcal{H}$ are locally absolutely continuous, in other words, absolutely continuous on each interval $[0, c]$ for $0<c<+\infty$;
 	\item [(ii)] $\ddot{w}(t)+\kappa(t) \dot{w}(t)+\rho(t) A(w(t))=0$ for almost every $t \in[0,+\infty)$;
 	\item [(iii)] $w(0)=w_{0}$ and $\dot{w}(0)=w_{1}$.
 \end{itemize} 
\end{dfn}	

The next theorem provides sufficient conditions ensuring the existence and uniqueness of solutions to the dynamical system \eqref{pt12}.

	\begin{thm} \label{dl3.1} Let $\rho, \kappa: \mathbb{R}^{+} \cup\{0\} \longrightarrow \mathbb{R}^{+} \cup\{0\}$ are Lebesgue measurable functions such that $\rho, \kappa \in L_{\text {loc }}^{1}\left(\mathbb{R}^{+} \cup\{0\}\right)$, let $T: H\longrightarrow H$ be an $\eta$-Lipschitz continuous operator and $g~:H \to H$ be a $\beta$-Lipschitz continuous operator, $f: K\longrightarrow \R$ be a proper, convex, l.s.c function. Then, the dynamical system \eqref{pt12} has a unique strong global solution for each $w_{0}, v_{0} \in H$.
		\end{thm} 	
		\begin{proof}
			Recall that the operator  $A: H \longrightarrow H$ is defined  by	
			$$
			A(w)= T(w)-P_{K}^{\gamma f}( T(w)-\gamma g(w)), \quad \forall w \in H .
			$$	
			Then, the dynamical system \eqref{pt12} reads as
			\begin{equation}\label{pt15} 
				\left\{\begin{array}{l}
					\ddot{w}(t)+\kappa(t) \dot{w}(t)+\rho(t) A w(t)=0,  \\
					w(0)=w_{0}, \quad \dot{w}(0)=v_{0} .
				\end{array}\right.
			\end{equation} 
			By the Cauchy-Schwarz inequality,  Part (ii) of Lemma \ref{lm2.1}, and the Lipschitz continuity of $ T$, it follows that $\forall w, w^* \in H$ and $\gamma>0$ :	
			\begin{align}\label{pt16} 
				\left\|A w-A w^*\right\| & =\left\|  T(w)-P_{K}^{\gamma f}\left( T(w)-\gamma g(w)\right)- T\left(w^*\right)+P_{K}^{\gamma f}\left( T(w^*)-\gamma g( w^*)\right) \right\| \notag \\
				& \leq \left\| T(w)- T(w^*) \right\|+ \left\|  P_{K}^{\gamma f}\left( T(w)-\gamma g(w)\right)-P_{K}^{\gamma f}\left( T(w^*)-\gamma g( w^*)\right) \right\| \notag \\
				& \leq \eta\|w-w^*\|+\|  T(w)-\gamma g(w)- T(w^*)+\gamma g( w^*)) \| \notag \\
				& \leq \eta\|w-w^*\|+\| T(w)- T(w^*)\|+\gamma\|g(w)-g\left(w^*\right)\| \notag \\
				& \leq \eta\|w-w^*\|+\eta\|w-w^*\|+\gamma\beta\|w-w^*\| \notag \\
				& \leq(2 \eta+\gamma\beta)\|w-w^*\|. 
			\end{align}	
			Hence, $A$ is $\tau$-Lipschitz continuous with modulus $\tau=2 \eta+\gamma\beta>0$.
			
			\noindent	Setting  $$\Psi: \mathbb{R}^{+} \cup\{0\} \longrightarrow H \times H, \Psi(t)=(w(t), \dot{w}(t)),$$ and	
			$$
			B: \mathbb{R}^{+} \cup\{0\} \times H \times H \longrightarrow H \times H, \quad B(t , x, y)=(y,-\kappa(t) y-\rho(t) A(x)) .
			$$	
			Then, \eqref{pt15} reduces to the following first-order dynamical system within the product space $H \times H$:	
			\begin{equation}\label{pt17} 
				\left\{\begin{array}{l}
					\dot{\Psi}(t)=B(t , \Psi(t)), \\
					\Psi(0)=\left(x_{0}, v_{0}\right).
				\end{array}\right.
			\end{equation}

			\noindent Using the Lipschitz continuity of $A$ and Cauchy-Schwarz inequality again, for $x, y, x^*, y^* \in H$ and $\forall t \geq 0$, we have 	
			$$
			\begin{aligned}
				\|B(t , x, y)-B(t , x^*, y^*)\|^2 & =\|y-y^*\|^{2}+\|\kappa(t)(y^*-y)+\rho(t)(A x^*-A x)\|^{2} \\
				& \leq \left(1+2 \kappa^{2}(t)\right)\|y-y^*\|^{2}+2 \tau^{2} \rho^{2}(t)\|x^*-x\|^{2} \\
				&\leq \Big(1+2 \kappa^{2}(t) +2 \tau^{2} \rho^{2}(t)\Big) \left(\|x^*-x\|^{2}+\|y-y^*\|^{2}\right)\\
				& \leq \left(1+2 \kappa^{2}(t)+2 \tau^{2} \rho^{2}(t)\right)\|(x, y)-(x^*, y^*)\|^2. 
			\end{aligned}
			$$	
			Consequently,  
			\begin{equation*}
				\|B(t , x, y)-B(t , x^*, y^*)\| \leq \Big(1+\sqrt{2} \kappa(t)+\tau \sqrt{2} \rho(t) \Big)\|(x, y)-(x^*, y^*)\|.
			\end{equation*}
			Because $\rho, \kappa \in L_{\text {loc }}^{1}\left(\mathbb{R}^{+} \cup\{0\}\right),$ it holds that $ 1+\sqrt{2} \kappa(t)+\tau \sqrt{2} \rho(t)$ is locally integrable. 
			
			We now show that	
			\begin{equation*}\label{pt18}
				B(\cdot, x, y) \in L^{1}\left([0, c], H \times H\right), \quad \forall x, y \in H, \forall c>0. 
			\end{equation*}	
			Indeed, let $c>0$ and $x, y \in H$, one has 
			$$
			\begin{aligned}
				\int_{0}^{c}\|B(t , x, y)\| dt& =\int_{0}^{c} \sqrt{\|y\|^{2}+\|\kappa(t) y+\rho(t) A x\|^{2}} dt\\
				& \leq \int_{0}^{c} \sqrt{\left(1+2 \kappa^{2}(t)\right)\|y\|^{2}+2 \rho^{2}(t)\|A x\|^{2}} dt\\
				& \leq \int_{0}^{c}((1+\sqrt{2} \kappa(t))\|y\|+\sqrt{2} \rho(t)\|A x\|) dt.
			\end{aligned}
			$$	
			Observe that $\|y\|$ and $\|A(x)\|$ are fixed. By the assumptions imposed on $\rho$ and $\kappa$, the above integral is finite. Applying the Cauchy–Lipschitz–Picard theorem to the first-order dynamical system, we obtain the existence and uniqueness of a global strong solution to system \eqref{pt17} (see \cite{15,16} for more details). Moreover, due to the equivalence of \eqref{pt12}, \eqref{pt15}, and \eqref{pt17}, the desired conclusion follows.
		\end{proof}

	\section{Exponential convergence}\label{hoitumu}
	The main result of this section is stated in the following theorem, whose proof relies on techniques introduced by Boț and Csetnek in \cite{37}. Further results on the connection between second-order dynamical systems and related problems are also available in \cite{37}.
	
	\begin{thm}\label{thm4.1}
		Let $K$ be a nonempty, closed, convex subset of $H$. Let $ T:H\longrightarrow H$ be $\eta$-Lipschitz continuous, $\lambda$-strongly monotone, let $g: H\longrightarrow H$ be a $\beta$-Lipschitz continuous, $f:K\longrightarrow \R\cup\{+\infty\}$ be a proper, convex, l.s.c function. Asume that the GIMVIP \eqref{pt10} admits a unique solution $w^*$. 
		Assume further that  
		\begin{enumerate}
			\item [(i)] The pair $(T, g)$ is $\alpha$-strongly coupled monotone.
			\item [(ii)] There exists $\gamma>0$ such that $a>0$  where  $a$ defined as in \eqref{dna}.
		\end{enumerate}		
		Let $\kappa, \mu:[0,+\infty) \longrightarrow[0,+\infty)$ be locally absolutely continuous functions satisfying for every $t \in[0,+\infty)$ that\\
		(i) $1<\kappa \leq \kappa(t) \leq a^{2} a_{1} \rho(t)+1$;\\
		(ii) $\dot{\kappa}(t) \leq 0$ and $\frac{d}{d t}\left(\frac{\kappa(t)}{\mu(t)}\right) \leq 0$;\\
		(iii) $\kappa^{2}(t)-\kappa(t)-\frac{2 \rho(t)}{a_{1}} \geq 0$.
		
		Consequently, the trajectory $w(t)$ generated by \eqref{pt12} converges exponentially to $w^*$ as $t \longrightarrow+\infty$, where $w^*$ is the unique solution of GIMVIP \eqref{pt10}.
	\end{thm}
	\begin{proof}
		
		By Theorem \ref{dl3.1}, the dynamical system \eqref{pt12} has a unique strong global solution. 
		
		Let $\Xi:[0, \infty) \longrightarrow[0, \infty)$ be defined by $\Xi(t)=\frac{1}{2}\left\|w(t)-w^*\right\|^{2}$. From dynamical system \eqref{pt12}, it holds that 
		$$\left\langle\ddot{w}(t), w(t)-w^*\right\rangle+\kappa(t)\left\langle\dot{w}(t), w(t)-w^*\right\rangle+\rho(t)\left\langle A(w(t)), w(t)-w^*\right\rangle=0,$$
		for almost every $t \in[0,+\infty)$. From the definition of $\Xi$, we have
		$$
		\dot{\Xi}(t)=\left\langle w(t)-w^*, \dot{w}(t)\right\rangle ; \quad \ddot{\Xi} (t)=\|\dot{w}(t)\|^{2}+\left\langle w(t)-w^*, \ddot{w}(t)\right\rangle .
		$$
		Therefore, 
		$$
		\ddot{\Xi} (t)+\kappa(t) \dot{\Xi}(t)+\rho(t)\left\langle A(w(t)), w(t)-w^*\right\rangle=\|\dot{w}(t)\|^{2}
		$$
		Using \eqref{pt23}, for any $w \in H $ we have\\
		$$a_{1}\left\|A(w)\right\|^{2} \leq\left\langle A(w), w-w^*\right\rangle.$$
		Consequently,
		$$\ddot{\Xi} (t)+\kappa(t) \dot{\Xi}(t)+a_{1} \rho(t)\left\|A(w(t))\right\|^{2}-\|\dot{w}(t)\|^{2} \leq 0.$$
		Using dynamical system \eqref{pt12} again, and substituting $A(w(t))=-\frac{1}{\rho(t)}(\ddot{w}(t)+\kappa(t) \dot{w}(t))$ we obtain
		$$
		\begin{aligned}
			& \ddot{\Xi} (t)+\kappa(t) \dot{\Xi}(t)+\frac{1}{2} a_{1} \rho(t)\left\|A(w(t))\right\|^{2}+\frac{a_{1}}{2 \rho(t)}\|\ddot{w}(t)+\kappa(t) \dot{w}(t)\|^{2} \\
			& -\|\dot{w}(t)\|^{2} \leq 0
		\end{aligned}
		$$
		Combining the above inequality with \eqref{pt23}, we get 
		\begin{align*}
			& \ddot{\Xi} (t)+\kappa(t) \dot{\Xi}(t)+a^{2} a_{1} \rho(t) \Xi(t)+\frac{a_{1}}{2 \rho(t)}\|\ddot{w}(t)\|^{2}+\left(\frac{a_{1} \kappa^{2}(t)}{2 \rho(t)}-1\right)\|\dot{w}(t)\|^{2} \\
			& +\frac{a_{1} \kappa(t)}{\rho(t)}\langle\ddot{w}(t), \dot{w}(t)\rangle \leq 0.
		\end{align*}
		By setting
		$$
		\begin{aligned}
			& m(t)=a^{2} a_{1} \rho(t), n(t)=\frac{a_{1} \kappa(t)}{2 \rho(t)}, p(t)=\frac{a_{1} \kappa^{2}(t)}{2 \rho(t)}-1 \\
			& v(t)=\|\dot{w}(t)\|^{2}, \text { for almost every } t \in[0,+\infty)
		\end{aligned}
		$$
		and noting that  $\frac{a_{1}}{2 \rho(t)}\|\ddot{w}(t)\|^{2} \geq 0$ and $\langle\ddot{w}(t), \dot{w}(t)\rangle=\frac{1}{2} \frac{d}{d t}\|\dot{w}(t)\|^{2}$ we get 
		\begin{equation}\label{pt18t} 
			\ddot{\Xi} (t)+\kappa(t) \dot{\Xi}(t)+m(t) \Xi(t)+n(t) \dot{v}(t)+p(t) v(t) \leq 0 . \tag{18}
		\end{equation} 
		Furthermore, one has 
		$$
		\begin{aligned}
			e^{t} \ddot{\Xi} (t)&=\frac{d}{d t}\left(e^{t} \dot{\Xi}(t)\right)-e^{t} \dot{\Xi}(t)=\frac{d}{d t}\left(e^{t} \dot{\Xi}(t)\right)-\frac{d}{d t}\left(e^{t} \Xi(t)\right)+e^{t} \Xi(t), \\
			\kappa(t) e^{t} \dot{\Xi}(t)&=\kappa(t) \frac{d}{d t}\left(e^{t} \Xi(t)\right)-\kappa(t) e^{t} \Xi(t), \\
			n(t) e^{t} \dot{v}(t)&=n(t) \frac{d}{d t}\left(e^{t} v(t)\right)-n(t) e^{t} v(t)
		\end{aligned}
		$$
		Multiplying both sides of \eqref{pt18t} by $e^{t}$ and using the above identities yields 
		\begin{align}\label{pt19t} 
			& \frac{d}{d t}\left(e^{t} \dot{\Xi}(t)\right)+(\kappa(t)-1) \frac{d}{d t}\left(e^{t} \Xi(t)\right)+(m(t)+1-\kappa(t)) e^{t} \Xi(t) \notag \\
			& +n(t) \frac{d}{d t}\left(e^{t} v(t)\right)+(p(t)-n(t)) e^{t} v(t) \leq 0 .
		\end{align}
		For almost every $t \in[0,+\infty)$, due to conditions (i) and (iii), it holds that 
		$$
		p(t)-n(t) \geq 0, m(t)+1-\kappa(t) \geq 0.
		$$
		Combining the above inequalities with \eqref{pt19t} yields 
		\begin{equation}\label{pt20t} 
			\frac{d}{d t}\left(e^{t} \dot{\Xi}(t)\right)+(\kappa(t)-1) \frac{d}{d t}\left(e^{t} \Xi(t)\right)+n(t) \frac{d}{d t}\left(e^{t} v(t)\right) \leq 0 . 
		\end{equation}
		Applying the following  identities 
		$$
		\begin{aligned}
			& n(t) \frac{d}{d t}\left(e^{t} v(t)\right)=\frac{d}{d t}\left[n(t) e^{t} v(t)\right]-\dot{n}(t) e^{t} v(t) \text { and } \\
			& (\kappa(t)-1) \frac{d}{d t}\left(e^{t} \Xi(t)\right)=\frac{d}{d t}\left[(\kappa(t)-1) e^{t} \Xi(t)\right]-\dot{\kappa}(t) e^{t} \Xi(t)
		\end{aligned}
		$$
		to \eqref{pt20t} gives 
		\begin{equation*}\label{pt21t} 
			\frac{d}{d t}\left(e^{t} \dot{\Xi}(t)\right)+\frac{d}{d t}\left[(\kappa(t)-1) e^{t} \Xi(t)\right]-\dot{\kappa}(t) e^{t} \Xi(t)+\frac{d}{d t}\left(n(t) e^{t} v(t)\right)-\dot{n}(t) e^{t} v(t) \leq 0 . 
		\end{equation*}
		Using assumption (ii), one has $\dot{\kappa}(t) \leq 0$ and $\dot{n}(t) \leq 0$. Hence,  from \eqref{pt21t} we obtain that 
		$$\frac{d}{d t}\left[e^{t} \dot{\Xi}(t)+(\kappa(t)-1) e^{t} \Xi(t)+n(t) e^{t} v(t)\right] \leq 0.$$
		Consequently, the special function 
		$$t \mapsto e^{t} \dot{\Xi}(t)+(\kappa(t)-1) e^{t} \Xi(t)+ n(t) e^{t} v(t),$$ is non-increasing. Hence, there exists a constant $b>0$ such that for almost every $t \in[0,+\infty)$ one has $$e^{t} \dot{\Xi}(t)+(\kappa(t)-1) e^{t} \Xi(t)+n(t) e^{t} v(t) \leq b.$$
		Since $n(t), v(t) \geq 0$ for almost every $t \in[0,+\infty)$, it holds that 
		$$ \dot{\Xi}(t)+(\kappa(t)-1) \Xi(t) \leq b e^{-t};$$ and  hence, $$\dot{\Xi}(t)+(\kappa-1) \Xi(t) \leq b e^{-t}.$$
		Now, multiplying both sides of the last inequality by $e^{(\kappa-1) t}>0$ we derive that for almost every $t \in[0,+\infty)$, 
		$$\frac{d}{d t}\left[e^{(\kappa-1) t} \Xi(t)\right] \leq be^{(\kappa-2) t}.$$
		We consider three cases:
		\begin{itemize}
			\item [(a)] if $\kappa=2$, then $0 \leq \Xi(t) \leq(b t+B(0)) e^{-t}$.
			\item [(b)] if $1<\kappa<2$, then $e^{(\kappa-1) t} \Xi(t) \leq \frac{b}{\kappa-2}\left[e^{(\kappa-2) t}-1\right]+B(0) \leq \frac{b}{2-\kappa}+B(0)$, which implies $\Xi(t) \leq\left[\frac{b}{2-\kappa}+B(0)\right] e^{-(\kappa-1) t}$;
			\item [(c)] if $\kappa>2$, then $e^{(\kappa-1) t} \Xi(t) \leq \frac{b}{\kappa-2}\left[e^{(\kappa-2) t}-1\right]+B(0) \leq \frac{b}{\kappa-2} e^{(\kappa-2) t}+B(0)$, which implies $\Xi(t) \leq \frac{b}{\kappa-2} e^{-t}+B(0) e^{-(\kappa-1) t} \leq\left(\frac{b}{\kappa-2}+B(0)\right) e^{-t}$.
		\end{itemize}
	By integrating both sides for each case, we always obtain that $w(t)$ converges exponentially to $w^*$.
	\end{proof} 
	\begin{remark} In \cite{Thanh} the author have been shown that there exist functions $\kappa(t)$ and $\rho(t)$ satisfying conditions (i)-(iii). 
		
	\end{remark}
	
	\section{ Discrete System with Linear Convergence Rate}\label{roirachoa}
	Using forward discrete time scheme to dynamical system \eqref{pt12} with a step size $z_{n}>0$, a relaxation variable $\rho_{n}>0$, a damping variable $\kappa_{n}>0$, along with initial points $w_{0}$ and $w_{1}$, we obtain the following iterative scheme:	

$$
\frac{w_{n+1}-2 w_{n}+w_{n-1}}{z_{n}^{2}}+\kappa \frac{w_{n}-w_{n-1}}{z_{n}}+\rho_{n}A(w_n)=0
$$
which is equivalent to 
$$
w_{n+1}=w_{n}+\left(1-\kappa_n z_{n}\right)\left(w_{n}-w_{n-1}\right)-\rho_{n} z_{n}^{2}A(w_n) .
$$
If $z_{n}=1, \kappa_n, \rho_{n}$ are positive constants, we can write the above scheme as

\begin{equation}\label{pt25v}
	\left\{\begin{array}{l}
		u_{n}:=w_{n}+(1-\kappa)\left(w_{n}-w_{n-1}\right) \\
		w_{n+1}=u_{n}-\rho  A(w_n) 
	\end{array}\right.
\end{equation} 
This is a projection algorithm with an inertial effect term $(1-\kappa)\left(w_{n}-w_{n-1}\right)$.

\noindent If $\kappa=1$, then \eqref{pt25v} collapses  to the projection algorithm 
\begin{equation}\label{pt26v} 
	w_{n+1}=w_{n}-\rho A(w_n) . 
\end{equation} 
Observe that this system arises from the discretization of the first-order dynamical system studied in \cite{JIT}.

To establish the convergence of the sequence generated by the discrete scheme, we first recall the difference operator and its fundamental properties. 
$$
\begin{array}{cl}
	w^{\Delta}(n):=w_{n+1}-w_{n}, & w^{\nabla}(n):=w_{n}-w_{n-1} \\
	w^{\Delta \nabla}:=\left(w^{\Delta}\right)^{\nabla}, & w^{\nabla \Delta}:=\left(w^{\nabla}\right)^{\Delta} .
\end{array}
$$
\begin{lem}\cite{Vuong2026}\label{lm4.1}
	The following statements  hold:  $$
	\begin{aligned}
		& \langle f, g\rangle^{\Delta}(n)=\left\langle f^{\Delta}(n), g_{n}\right\rangle+\left\langle z_{n}, g^{\Delta}(n)\right\rangle+\left\langle f^{\Delta}(n), g^{\Delta}(n)\right\rangle, \\
		& \langle f, g\rangle^{\nabla}(n)=\left\langle f^{\nabla}(n), g_{n}\right\rangle+\left\langle z_{n}, g^{\nabla}(n)\right\rangle-\left\langle f^{\nabla}(n), g^{\nabla}(n)\right\rangle, \\
		& w^{\Delta \nabla}(n)=w^{\nabla \Delta}(n)=w_{n+1}-2 w_{n}+w_{n-1}=w^{\Delta}(n)-w^{\nabla}(n) .
	\end{aligned}
	$$
\end{lem}
By applying the difference operations, \eqref{pt25v} can be rewritten as follows 
\begin{equation}\label{pt27v}
	w^{\Delta \nabla}(n)+\kappa w^{\nabla}(n)=-\rho A(w_n)  . 
\end{equation}
We now define 
$$
\mu_{n}:=\left\|w^{\Delta}(n)\right\|^{2}, \quad y_n:=\left\|w^{\nabla}(n)\right\|^{2}, \quad x_{n}:=\left\|w_{n}-w^{*}\right\|^{2} .
$$

\begin{thm}  Let $K$ be a nonempty, closed, convex subset of $H$. Let $ T:H\longrightarrow H$ be $\eta$-Lipschitz continuous, $\lambda$-strongly monotone, let $g: H\longrightarrow H$ be a $\beta$-Lipschitz continuous, $f:K\longrightarrow \R\cup\{+\infty\}$ be a proper, convex, l.s.c function. Assume further that $(T, g)$ is $\alpha$-strongly couple monotone and that there exists $\gamma>0$ such that $a>0$ where  $a$ defined as in \eqref{dna}
	Moreover, the coefficients $\kappa$ and $\rho$ are assumed to satisfy the following conditions.
\begin{enumerate}
	\item [(A1)]$0<\kappa<1$
	\item [(A2)] $0<\rho<a_{1} \cdot \min \left\{\frac{1-\kappa}{4}, \frac{\kappa^{2}}{4-\kappa}\right\}$
\end{enumerate}
Then, the sequence ${w_n}$ generated by \eqref{pt27v} converges linearly to the unique solution of the GIMVIP \eqref{pt10}. 	
\end{thm} 

\begin{proof}  From (A1)-(A2) one has 
	\begin{equation}\label{pt28v} 
		\frac{a_{1}}{\rho}(1-\kappa) \geq 4. 
	\end{equation}
	In addition, 
	$$
	\kappa\left(\frac{a_{1} \kappa}{\rho}+1\right)-4>0
	$$
	Thus, one has  
	$$
	\begin{gathered}
		\lim _{\varepsilon \rightarrow 1^{+}}\left[\varepsilon\left(\kappa\left(\frac{a_{1} \kappa}{\rho}+1\right)-4\right)-(\varepsilon-1)\left(\frac{a_{1} \kappa}{\rho}+1\right)\right]=\kappa\left(\frac{a_{1} \kappa}{\rho}+1\right)-4>0 \\
		\lim _{\varepsilon \rightarrow 1^{+}}\left[\varepsilon a \rho-\kappa \varepsilon(\varepsilon-1)+(\varepsilon-1)^{2}\right]=a \rho>0 \\
		\lim _{\varepsilon \rightarrow 1^{+}}\left[1-\varepsilon^{2}(1-\kappa)\right]=1-(1-\kappa)=\kappa>0
	\end{gathered}
	$$
	As a result, there exists $\varepsilon>1$ satisfying  the followings
	\begin{gather}
		\varepsilon\left(\kappa\left(\frac{a_{1} \kappa}{\rho}+1\right)-4\right)-(\varepsilon-1)\left(\frac{a_{1} \kappa}{\rho}+1\right)>0  \label{pt29v}\\
		\varepsilon a \rho-\kappa \varepsilon(\varepsilon-1)+(\varepsilon-1)^{2}>0  \label{pt30v}\\
		1-\varepsilon^{2}(1-\kappa)>0 \label{pt31v}
	\end{gather}
	Let 
	$$
	C_{0}:=\kappa\left(\frac{a_{1}}{\rho} \kappa+1\right)-4, \quad C_{1}:=\frac{a_{1}}{\rho} \kappa+1.
	$$
	One has 
	$$
	x^{\Delta \nabla}(n)+\kappa x^{\nabla}(n)= -2 \rho\left\langle A(w_n), w_{n}-w^{*}\right\rangle+2 \mu_n-\kappa y_{n}-y^{\Delta}(n). 
	$$
	By \eqref{pt23} and \eqref{pt23n}, it follows that 
	$$
	x^{\Delta \nabla}(n)+\kappa x^{\nabla}(n) \leq-a_{1} \rho\left\|A(w_n)\right\|^{2}-a \rho x_{n}+2 \mu_n-\kappa y_{n}-y^{\Delta}(n).
	$$
	Then
	$$
	\begin{gathered}
		x^{\Delta \nabla}(n)+\kappa x^{\nabla}(n)+a \rho x_{n}+\kappa y_{n}+y^{\Delta}(n) \\
		\leq-a_{1} \rho\left\|A(w_n)\right\|^{2}+2 \mu_n \\
		=-\frac{a_{1}}{\rho}\left\|w^{\Delta \nabla}(n)+\kappa w^{\nabla}(n)\right\|^{2}+2 \mu_n \\
		=-\frac{a_{1}}{\rho}\left(\left\|w^{\Delta \nabla}(n)\right\|^{2}+\kappa^{2} y_{n}+2 \kappa\left\langle w^{\Delta \nabla}(n), w^{\nabla}(n)\right\rangle\right)+2 \mu_n \\
		=-\frac{a_{1}}{\rho}\left(\left\|w^{\Delta \nabla}(n)\right\|^{2}(1-\kappa)+\kappa^{2} y_{n}+\kappa y^{\Delta}(n)\right)+2 \mu_n .
	\end{gathered}
	$$
	This implies that 
	$$
	\begin{gathered}
		x^{\Delta \nabla}(n)+\kappa x^{\nabla}(n)+a \rho x_{n}+\kappa\left(\frac{a_{1}}{\rho} \kappa+1\right) y_{n}+\left(\frac{a_{1}}{\rho} \kappa+1\right) y^{\Delta}(n) \\
		\leq-\frac{a_{1}}{\rho}(1-\kappa)\left\|w^{\Delta \nabla}(n)\right\|^{2}+2 \mu_n.
	\end{gathered}
	$$
	Since 
	$$
	\begin{aligned}
		\mu_n=\| w^{\Delta}(n) & -w^{\nabla}(n)+w^{\nabla}(n)\left\|^{2}=\right\| w^{\Delta \nabla}(n)+w^{\nabla}(n) \|^{2} \\
		& \leq 2\left(\left\|w^{\Delta \nabla}(n)\right\|^{2}+\left\|w^{\nabla}(n)\right\|^{2}\right), 
	\end{aligned}
	$$
	and by \eqref{pt28v} we derive that 
	\begin{equation*}  
		\begin{gathered}
			x^{\Delta \nabla}(n)+\kappa x^{\nabla}(n)+a \rho x_{n}+C_{0} y_{n}+C_{1} y^{\Delta}(n) \\
			\leq\left(4-\frac{a_{1}}{\rho}(1-\kappa)\right)\left\|w^{\Delta \nabla}(n)\right\|^{2} \leq 0 
		\end{gathered}
	\end{equation*} 
	Multiplying both sides by $\varepsilon^{n+1}$ and then using Lemma \ref{lm4.1} one obtains 
	$$
	\begin{gathered}
		0 \geq \varepsilon^{n+1}\left(x^{\Delta \nabla}(n)+\kappa x^{\nabla}(n)+a \rho x_{n}+C_{0} y_{n}+C_{1} y^{\Delta}(n)\right) \\
		=\left(\varepsilon^{n} x^{\nabla}\right)^{\Delta}(n)+(\varepsilon \kappa-\varepsilon+1)\left(\varepsilon^{n+1} x\right)^{\nabla}(n)+\varepsilon^{n} x_{n}\left[\varepsilon a \rho-\kappa \varepsilon(\varepsilon-1)+(\varepsilon-1)^{2}\right] \\
		+C_{1}\left(\varepsilon^{n} y\right)^{\Delta}(n)+\varepsilon^{n} y_{n}\left[\varepsilon C_{0}-(\varepsilon-1) C_{1}\right]
	\end{gathered}
	$$
	By \eqref{pt29v}-\eqref{pt30v}, it follows from the inequality that 
	$$
	0 \geq\left(\varepsilon^{n} x^{\nabla}\right)^{\Delta}(n)+(\varepsilon \kappa-\varepsilon+1)\left(\varepsilon^{n+1} x\right)^{\nabla}(n)+C_{1}\left(\varepsilon^{n} y\right)^{\Delta}(n). 
	$$
	By summing the preceding inequality over $n=1,\dots,m$, we get 
	$$
	M_{1} \geq \varepsilon^{m+1} x^{\nabla}(m+1)+(\varepsilon \kappa-\varepsilon+1) \varepsilon^{m+1} x_{m}+C_{1} \varepsilon^{m+1} y_{m+1}
	$$
	where $M_{1}$ is some positive constant. Since $C_{1}>0$, we derive that 
	$$
	\begin{aligned}
		M_{1} & \geq \varepsilon^{m+1} x^{\nabla}(m+1)+(\varepsilon \kappa-\varepsilon+1) \varepsilon^{m+1} x_{m} \\
		& =\varepsilon^{m+1} x^{\Delta}(m)+(\varepsilon \kappa-\varepsilon+1) \varepsilon^{m+1} x_{m} \\
		& =\left(\varepsilon^{m} x_{m}\right)^{\Delta}+\varepsilon^{m} x_{m}\left[\varepsilon^{2}(\kappa-1)+1\right] \\
		& \geq\left(\varepsilon^{m} x_{m}\right)^{\Delta}
	\end{aligned}
	$$
	where the last inequality is due to \eqref{pt31v}. Summing the inequality for $m=1$ to $m=k$ yields 
	$$
	M_{1} k+\varepsilon x_{1} \geq \varepsilon^{k+1} x_{k+1}
	$$
	meaning that the sequence $\left\{w_{n}\right\}$ generated by \eqref{pt27v} converges linearly to the unique solution of the GIMVIP \eqref{pt10}.
\end{proof}

	\section{Numerical Experiment}\label{vdso} 
	In this section, we present a numerical example to illustrate the convergence behavior of the iterative algorithms \eqref{pt25v} and \eqref{pt26v}. The numerical simulations were carried out in Python using Google Colab.

\begin{exm}
Let $d=1$ and $K=[0,+\infty)$. Define
\[
g(w)=\frac{w}{2}, \qquad 
T(w)=\frac{3w}{4}, \qquad 
f(w)=w^2+2w+1, \quad w\in\mathbb{R}.
\]
It is straightforward to verify that $g$ is Lipschitz continuous with constant $\beta=\tfrac{1}{2}$ and strongly monotone with constant $\zeta=\tfrac{1}{2}$, while $T$ is Lipschitz continuous with constant $\eta=\tfrac{3}{4}$ and strongly monotone with constant $\lambda=\tfrac{3}{4}$. Consequently, the pair $(T,g)$ is strongly coupled monotone with constant $\alpha=\tfrac{3}{8}$.

As shown in \cite{Nam}, these parameters satisfy condition~(ii) of Theorem~\ref{thm4.1}. We further set $\gamma=1.4$ and then $a>0$ where $a$ is defined in~\eqref{dna}. A direct verification shows that $w^*=0$ is a solution of~\eqref{pt10}.

We now compare the proposed inertial (second-order) algorithm with the corresponding first-order method for different values of $\kappa$ and $\rho$. Unless otherwise stated, the initial values are chosen as
\[
w_0=w_1=100.
\]

\begin{table}[htbp]
	\centering
	\caption{Performance comparison of inertial and non-inertial methods with $\rho=0.09$}
	\label{tab:inertial2}
	\begin{tabular}{|l|c|}
		\hline
		Method & Error value \\ \hline
		Inertial ($\kappa = 0.1$)  & $1.5258126167 \times 10^{-10}$ \\ \hline
		Inertial ($\kappa = 0.59$) & $3.5502140699 \times 10^{-28}$ \\ \hline
		Inertial ($\kappa = 0.9$)  & $8.8911678718 \times 10^{-16}$ \\ \hline
		Standard (No inertia)     & $6.6935161972 \times 10^{-14}$ \\ \hline
	\end{tabular}
\end{table}

Figure~\ref{fg1} illustrates the convergence behavior of Algorithms~\ref{pt25v} and \ref{pt26v} for different values of $\kappa$ with $\rho=0.09$.

\begin{figure}[H]
	\centering
	\includegraphics[width=0.75\linewidth]{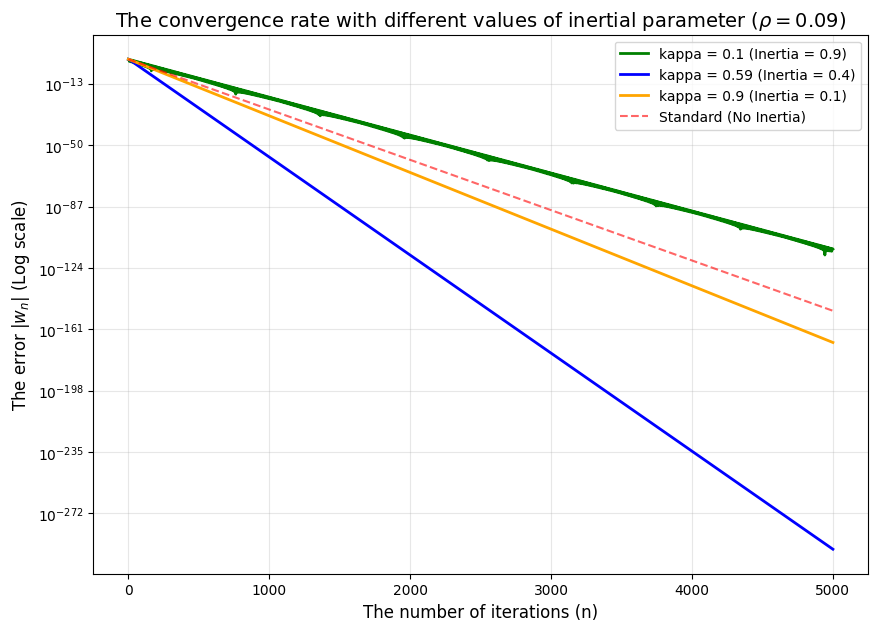}
	\caption{Convergence rate of Algorithms~\ref{pt25v} and \ref{pt26v} for $\rho=0.09$.}
	\label{fg1}
\end{figure}

\begin{table}[htbp]
	\centering
	\caption{Performance comparison of inertial and non-inertial methods with $\rho=0.0019$}
	\label{tab:inertial}
	\begin{tabular}{|l|c|}
		\hline
		Method & Error value \\ \hline
		Inertial ($\kappa = 0.1$)  & $1.5856421935 \times 10^{-35}$ \\ \hline
		Inertial ($\kappa = 0.59$) & $5.5334312632 \times 10^{-4}$  \\ \hline
		Inertial ($\kappa = 0.9$)  & $3.6336573415 \times 10^{-2}$  \\ \hline
		Standard (No inertia)     & $8.0263714274 \times 10^{-2}$  \\ \hline
	\end{tabular}
\end{table}

The convergence profiles obtained with different values of $\kappa$ and $\rho=0.0019$ after $5000$ iterations are depicted in Figure~\ref{fg2}.

\begin{figure}[H]
	\centering
	\includegraphics[width=0.75\linewidth]{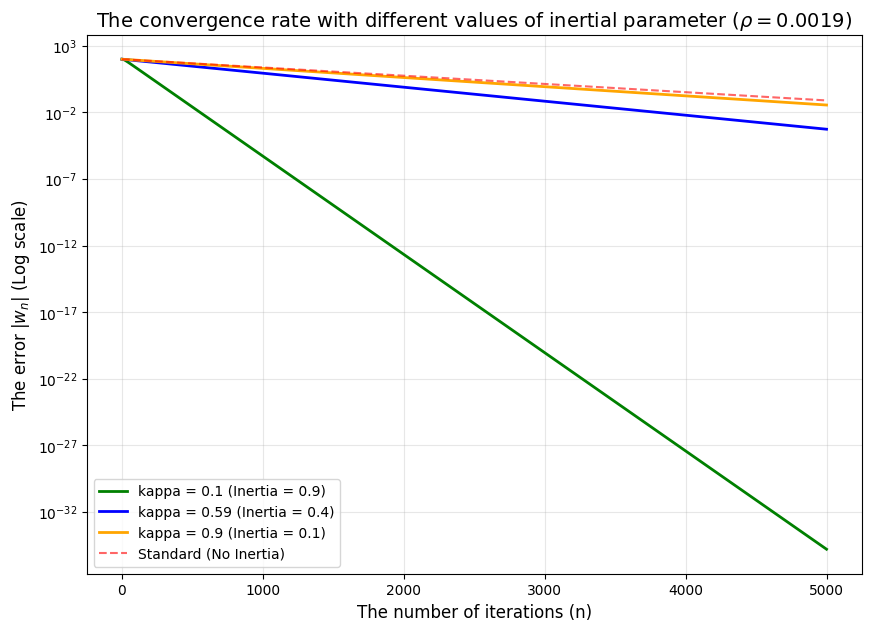}
	\caption{Convergence rate of Algorithms~\ref{pt25v} and \ref{pt26v} for $\rho=0.0019$.}
	\label{fg2}
\end{figure}

\end{exm} 

	\section{ Conclusion}\label{kl} 
In this paper, we have investigated a second-order inertial dynamical system associated with the generalized inverse mixed variational inequality problem (GIMVIP). Under mild assumptions, we established the global existence and uniqueness of solution trajectories for the proposed continuous-time system. By constructing an appropriate Lyapunov functional, we proved that the trajectories converge exponentially to the unique solution of the GIMVIP, thereby highlighting the strong stability properties of the considered second-order dynamics.

Furthermore, motivated by the continuous model, we derived an accelerated discrete algorithm and rigorously established its linear convergence rate. This result provides a clear connection between the theoretical analysis of second-order dynamical systems and practical numerical algorithms for solving generalized inverse variational inequality problems. Numerical experiments were presented to support the theoretical findings and to demonstrate that the proposed second-order approach exhibits faster error decay and more robust performance when compared with classical first-order methods.

The results of this work indicate that second-order inertial systems constitute an effective and theoretically sound framework for solving generalized inverse models, preserving accelerated convergence behavior even in highly generalized settings.

	\section*{Declaration of interests}
	The authors declare that they have no known competing financial interests or personal relationships that could have appearedzo influence the work reported in this paper.
	
	\section*{Conflict of interest}
	The authors declare no conflicts of interest in this paper.

\end{document}